\documentclass{amsart}
\usepackage{graphicx}
\usepackage{ulem}
\usepackage{bm}
\usepackage{color}
\usepackage{amssymb, amsmath, amsthm}
\usepackage{bm}
\theoremstyle{plain}
\usepackage[aboveskip=-5pt,position=bottom]{caption}
\usepackage{natbib}
\usepackage{url}
\usepackage{enumerate}
\usepackage{rotating}
\usepackage{booktabs,arydshln}
\usepackage{caption} 
\usepackage[table,x11names]{xcolor}

\newcommand{\mycmidrule}{\cmidrule[\heavyrulewidth]}

\makeatletter
\def\adl@drawiv#1#2#3{%
        \hskip.5\tabcolsep
        \xleaders#3{#2.5\@tempdimb #1{1}#2.5\@tempdimb}%
                #2\z@ plus1fil minus1fil\relax
        \hskip.5\tabcolsep}
\newcommand{\cdashlinelr}[1]{%
  \noalign{\vskip\aboverulesep
           \global\let\@dashdrawstore\adl@draw
           \global\let\adl@draw\adl@drawiv}
  \cdashline{#1}
  \noalign{\global\let\adl@draw\@dashdrawstore
           \vskip\belowrulesep}}
\makeatother

\usepackage{accents}

\newtheorem{lemma}{Lemma}[section]

\newtheorem{theorem}[lemma]{Theorem}
\newtheorem{cor}[lemma]{Corollary}

\newtheorem{exam}[lemma]{\normalfont \scshape
 Example}

\newtheorem{defi}[lemma]{Definition}

\newcommand{\R}{\mathbb{R}}
\newcommand{\N}{\mathbb{N}}

\newcommand{\norm}[1]{\left\Vert#1\right\Vert}
\newcommand{\dnormf}[1]{\wr\hspace*{-2.5pt}\wr #1 \wr\hspace*{-2.5pt}\wr}
\newcommand{\abs}[1]{\left\vert#1\right\vert}
\newcommand{\set}[1]{\left\{#1\right\}}

\newcommand{\bfx}{\bm{x}}
\newcommand{\bfzero}{\bm{0}}

\newcommand{\bfone}{\bm{1}}
\newcommand{\bfa}{\bm{a}}
\newcommand{\bfb}{\bm{b}}

\newcommand{\bfe}{\bm{e}}
\newcommand{\bfM}{\bm{M}}

\newcommand{\bfr}{\bm{r}}

\newcommand{\bft}{\bm{t}}
\newcommand{\bfU}{\bm{U}}
\newcommand{\bfu}{\bm{u}}

\newcommand{\bfv}{\bm{v}}
\newcommand{\bfV}{\bm{V}}

\newcommand{\bfX}{\bm{X}}
\newcommand{\bfY}{\bm{Y}}
\newcommand{\bfy}{\bm{y}}
\newcommand{\bfZ}{\bm{Z}}
\newcommand{\bfz}{\bm{z}}

\usepackage[doublespacing]{setspace}

\allowdisplaybreaks[4]

\usepackage{hyperref}
\hypersetup{colorlinks,%
citecolor=blue,%
filecolor=green,%
linkcolor=red,%
urlcolor=violet,%
}

\begin{document}

\title[Generalized Pareto Copulas]{Generalized Pareto Copulas: A Key to Multivariate Extremes}

\author{Michael Falk\and Simone Padoan\and Florian Wisheckel}
\address[1,3]{Institute of Mathematics, University of W\"{u}rzburg, W\"{u}rzburg, Germany}
\email{michael.falk@uni-wuerzburg.de}\email{florian.wisheckel@uni-wuerzburg.de}
\address[2]{Department of Decision Sciences,
Bocconi University of Milan,  Milano, Italy}
\email{simone.padoan@unibocconi.it}



\begin{abstract}
This paper reviews generalized Pareto copulas (GPC), which turn out to be a key to multivariate extreme value theory.
Any GPC can be represented in an easy analytic way using a particular type of norm on $\R^d$, called $D$-norm. The characteristic property of a GPC is its exceedance stability.

GPC might help to end the debate: What is a \emph{multivariate} generalized Pareto distribution?
We present an easy way how to simulate data from an arbitrary GPC and, thus, from an arbitrary generalized Pareto distribution.

As an application we derive  nonparametric estimates of the probability that a  random vector, which follows a GPC, exceeds a high threshold, together with confidence intervals. A case study  on joint exceedance probabilities for air pollutants completes the paper.

\end{abstract}

\maketitle


\section{Introduction and Preliminaries}

Let $\bfX=(X_1,\dots,X_d)$ be a random vector (rv), whose distribution function (df) is in the domain of attraction of a multivariate non degenerate df $G$, denoted by $F\in\mathcal D(G)$, i.e., there exist vectors $\bfa_n>\bfzero\in\R^d$, $\bfb_n\in\R^d$, $n\in\N$, such that
\begin{equation}\label{cond:F in max-domain of attraction}
F^n(\bfa_n\bfx+\bfb_n)\to_{n\to\infty} G(\bfx),\qquad \bfx\in\R^d.
\end{equation}
All operations on vectors $\bfx,\bfy$ such as $\bfx+\bfy$, $\bfx\bfy$ etc. are meant componentwise.

The limit df $G$ is necessarily max-stable, i.e., there exist vectors $\bfa_n>\bfzero\in\R^d$, $\bfb_n\in\R^d$, $n\in\N$, such that
\[
G^n(\bfa_n\bfx+\bfb_n) = G(\bfx),\qquad \bfx\in\R^d.
\]
A characterization of multivariate max-stable df was established by \citet{dehar77} and \citet{vatan85}; for an introduction to multivariate extreme value theory see, e.g., \citet[Chapter 4]{fahure10}.

The univariate margins $G_i$, $1\le i\le d$, of a multivariate max-stable df $G$ belong necessarily to the family of univariate max-stable df, which is a parametric family $\set{G_\alpha:\,\alpha\in\R}$ with $$\begin{array}{ll}
G_\alpha(x)=\left\{
\begin{array}{lll}
\exp\big(-(-x)^\alpha\big), & \quad x\le 0,\\[1ex]
1,                       & \quad x>0,
\end{array}\right. &\qquad \mbox{for }\alpha>0,\\[4ex]
G_\alpha(x)=\left\{
\begin{array}{lll}
0,                         &\quad x\le 0,\\[1ex]
\exp(-x^\alpha),\hspace*{8.3mm} & \quad x>0,
\end{array}\right. &\qquad \mbox{for }\alpha<0,
\end{array}
$$
and
\begin{equation}\label{defn:of univariate EVD}
G_0(x):=\exp(-e^{-x}),\qquad x\in\R,
\end{equation}
being the family of reverse Weibull\index{Distribution!reverse Weibull}, Fr\'{e}chet\index{Distribution!Fr\'{e}chet} and  Gumbel distributions\index{Distribution!Gumbel}. Note that $G_1(x)=\exp(x)$, $x\le 0$, is the standard \emph{negative exponential} df. We refer, e.g., to \citet[Section 2.3]{gal87} or \citet[Chapter 1]{resn87}.

By Sklar's theorem (\citet{sklar59, sklar96}), there exists a rv $\bfU=(U_1,\dots,U_d)$ with the property that each component $U_i$ follows the uniform distribution on $(0,1)$, such that
\[
\bfX=_D \left(F_1^{-1}(U_1),\dots,F_d^{-1}(U_d)\right),
\]
where $F_i$ is the df of $X_i$ and $F_i^{-1}(u)=\inf\set{t\in\R:\,F_i(t)\ge u}$, $u\in (0,1)$, is the common generalized inverse or quantile function of $F_i$, $1\le i\le d$. By $=_D$ we denote equality in distribution.

The rv $\bfU$, therefore, follows a \emph{copula}, say $C_F$. If $F$ is continuous, then the copula $C_F$ is uniquely determined and given by $C_F(\bfu)=F\left(F_1^{-1}(u_1),\dots,F_d^{-1}(u_d)\right)$, $\bfu=(u_1,\dots,u_d)\in(0,1)^d$.

\citet{deheu84} and \citet{gal87} showed that
$F\in\mathcal D(G)$ iff this is true for each univariate margin $F_i$ \emph{and} for the copula $C_F$. Precisely, they established the following result.

\begin{theorem}[\citet{deheu84}, \citet{gal87}]\label{theo:deheuvels and galambos}
The df $F$ satisfies $F\in\mathcal D(G)$ iff this is true for the univariate margins of $F$ together with the convergence of the copulas:
\begin{equation}\label{cond:original_condition_on_copula}
C_F^n\left(\bfu^{1/n}\right)\to_{n\to\infty}C_G(\bfu)=G\left(\left(G_i^{-1}(u_i)\right)_{i=1}^d\right),
\end{equation}
$\bfu=(u_1,\dots,u_d)\in(0,1)^d$, where $G_i$ denotes the $i$-th margin of $G$, $1\le i\le d$.
\end{theorem}

Let $\bfU^{(1)},\bfU^{(2)},\dots$ be independent copies of the rv $\bfU$, which follows the copula $C_F$. Then the copula $C_{\bfM_n}$ of
\[
\bfM_n:= \max_{1\le i\le n}\bfU^{(i)}
\]
is $C_F^n\left(\bfu^{1/n}\right)$, where the maximum is also taken componentwise. The df of $\bfM_n$ is $C_F^n$ and, thus, we have
\[
C_F^n\left(\bfu^{1/n} \right) = C_{\bfM_n}(\bfu) = C_{C_F^n}(\bfu),\qquad \bfu\in[0,1]^d.
\]
Therefore, condition \eqref{cond:original_condition_on_copula} actually means pointwise convergence of the copulas
\[
C_{\bfM_n}(\bfu)\to_{n\to\infty} C_G(\bfu),
\]
where $C_G(\bfu)=G\left(\left(G_i^{-1}(u_i)\right)_{i=1}^d\right)$, $\bfu\in(0,1)^d$, is the copula of $G$. This is an \emph{extreme value copula}. Note that each margin $G_i$ of $G$ is continuous, which is equivalent with the continuity of $G$ (see, e.g., \citet[Lemma 2.2.6]{reiss89}).

Elementary arguments imply that condition \eqref{cond:original_condition_on_copula} is equivalent with the condition
\begin{equation}\label{cond:copula in domain of attraction}
C_F^n\left(\bfone+\frac{\bfy}n\right)\to_{n\to\infty} G^*(\bfy):=C_G(\exp(\bfy)),\qquad \bfy\le\bfzero\in\R^d,
\end{equation}
where $\bfone=(1,\dots,1)\in\R^d$ and $G^*(\bfy)$, $\bfy\le\bfzero\in\R^d$, defines a max-stable df with standard negative exponential margins $G^*_i(y)=\exp(y)$, $y\le 0$, $1\le i\le d$. Such a max-stable df will be called a \emph{standard} one, abbreviated by SMS (standard max-stable).

While the condition on the univariate margins $F_i$ in Theorem \ref{theo:deheuvels and galambos} addresses \emph{univariate} extreme value theory,  condition \eqref{cond:original_condition_on_copula} on the copula $C_F$ means by the equivalent condition \eqref{cond:copula in domain of attraction}  that the copula $C_F$ is in the domain of attraction of a multivariate SMS df:
\[
C_F^n\left(\bfone+\frac{\bfy}n\right)=P(n(\bfM_n-\bfone)\le \bfy)\to_{n\to\infty}G^*(\bfy),\qquad \bfy\le\bfzero\in\R^d.
\]

Let $C$ be an arbitrary copula on $\R^d$. Then condition \eqref{cond:F in max-domain of attraction} becomes
\begin{align*}
C\in\mathcal D(G)
\iff C^n(\bfa_n\bfx+\bfb_n)\to_{n\to \infty} G(\bfx),\qquad  \bfx\in\R^d,
\end{align*}
where the norming constants $\bfa_n,\bfb_n$ are determined by the univariate margins $C_i$ of $C$, i.e., the uniform distribution on $(0,1)$:
With $a_n=1/n$, $b_n=1$ we obtain for large $n$
\begin{align*}
C_i(a_nx+b_n)^n&=\left(1+\frac xn \right)^n
\to_{n\to\infty}\exp(x),\qquad x\le 0.
\end{align*}

We therefore obtain the conclusion: If a copula $C$ satisfies $C\in\mathcal D(G)$, then the limiting df $G$ has necessarily standard negative exponential margins:
\[
G_i(x)=\exp(x),\qquad x\le 0,\;1\le i\le d,
\]
i.e., the limiting df $G$ is necessarily a SMS df.

As a consequence we obtain that \emph{multivariate} extreme value theory actually means extreme value theory for \emph{copulas}.

This paper is organized as follows. In the next section we introduce $D$-norms, which turn out to be a common thread in multivariate extreme value theory. Using the concept of $D$-norms, we introduce in Section \ref{sec:generalized Pareto copulas} generalized Pareto copulas (GPC). The characteristic property of a GPC is its excursion or exceedance stability, established in Theorem \ref{propo:characteristic_property_of_GPC}. The family of GPC together with the well-known set of univariate generalized Pareto distributions (GPD) enables the definition of multivariate GPD in Section \ref{sec:multivariate generalized Pareto distributions}. As the set of univariate GPD equals the set of univariate non degenerate exceedance stable distributions, its extension to higher dimensions via a GPC and GPD margins is an obvious idea. $\delta$-neighborhoods of a GPC are introduced in Section \ref{sec:delta-neighborhoods of GPC}. The normal copula is a prominent example. Among others we show how to simulate data, which follow a copula from such a $\delta$-neighborhood. In Section \ref{sec:estimaton of exceedance probability} we show how our findings on GPC can be used to estimate exceedance probabilities above high thresholds, including confidence intervals. A case study in Section \ref{sec:case study} on joint exceedance probabilities for air pollutants such as  ozone, nitrogen dioxide, nitrogen oxide, sulphur dioxide and particulate matter, completes the paper.

\section{Introducing D-Norms}\label{sec:D-norms}
A crucial characterization of SMS df due to \citet{balr77}, \citet{dehar77}, \citet{pick81} and \citet{vatan85} can be formulated as follows; see \citet[Section 4.4]{fahure10}.

\begin{theorem}[\citet{balr77}, \citet{dehar77}, \citet{pick81}, \citet{vatan85}]
  A df $G$ on $\R^d$ is an SMS df iff there exists a norm $\norm\cdot$ on $\R^d$ such that
\begin{equation}\label{eqn:characterization of sms df}
G(\bfx)=\exp(-\norm{\bfx}),\qquad \bfx\le \bfzero\in\R^d.
\end{equation}
\end{theorem}

Elementary arguments imply the following consequence.

\begin{cor}\label{coro:characterization of domain of attraction of copula}
A copula $C$ satisfies $C\in\mathcal D(G)$ iff there exists a norm $\norm\cdot$ on $\R^d$ such that
\begin{equation}\label{eqn:expansion of copula in domain of attraction}
C(\bfu) = 1-\norm{\bfone-\bfu}+ o(\norm{\bfone-\bfu})
\end{equation}
as $\bfu\to\bfone\in\R^d$, uniformly for $\bfu\in[0,1]^d$.
\end{cor}

Those norms, which can appear in the preceding result, can be characterized. Any norm $\norm\cdot$ in equation \eqref{eqn:characterization of sms df} or \eqref{eqn:expansion of copula in domain of attraction} is necessarily of the following kind: There exists a rv $\bfZ=(Z_1,\dots,Z_d)$, whose components satisfy
\[
Z_i\ge 0,\quad E(Z_i)=1,\qquad 1\le i\le d,
\]
with
\[
\norm{\bfx}=E\left(\max_{1\le i\le d}\left(\abs{x_i}Z_i\right) \right)=: \norm{\bfx}_D,
\]
$\bfx=(x_1,\dots,x_d)\in\R^d$.

Such a norm $\norm\cdot_D$ is called \emph{$D$-norm}, with \emph{generator} $\bfZ$. The additional index $D$ means \emph{dependence}. $D$-norms were first mentioned in \citet[equation (4.25)]{fahure04} and more elaborated in \citet[Section 4.4]{fahure10}.
Examples are:
\begin{itemize}
\item $\norm{\bfx}_\infty=\max_{1\le i\le d}\abs{x_i}$, with generator $\bfZ=(1,\dots,1)\in\R^d$,
\item $\norm{\bfx}_1=\sum_{i=1}^d\abs{x_i}$, with generator $\bfZ$  being a random permutation of the vector $(d,0,\dots,0)\in\R^d$,
\item each logistic norm $\norm{\bfx}_p=\left(\sum_{i=1}^d\abs{x_i}^p\right)^{1/p}$, $p\in(1,\infty)$, with generator $\bfZ=(Z_1,\dots,Z_d)=(Y_1,\dots,Y_d)/\Gamma(1-1/p)$, $Y_1,\dots,Y_d$ iid Fr\'{e}chet-distributed rv with parameter $p$, where $\Gamma$ denotes the usual gamma function.
\item Let the rv $\bfX=(X_1,\dots,X_d)$ follow a multivariate normal distribution with mean vector zero, i.e., $E(X_i)=0$, $1\le i\le d$, and covariance matrix $\Sigma=(\sigma_{ij})_{1\le i,j\le d}=(E(X_iX_j))_{1\le i,j\le d}$. Then $\exp(X_i)$ follows a log-normal distribution\index{Distribution!log-normal} with mean $\exp(\sigma_{ii}/2)$, $1\le i\le d$, and, thus,
\begin{equation*}\label{eqn:definition of generator of Huesler-Reiss D-Norm}
\bfZ=(Z_1,\dots,Z_d):= \left(\exp\left(X_1-\frac{\sigma_{11}}2\right),\dots, \exp\left(X_d-\frac{\sigma_{dd}}2\right) \right)
\end{equation*}
is the generator of a $D$-norm, called \emph{H\"{u}sler-Reiss $D$-norm}\index{H\"{u}sler-Reiss $D$-norm}. This norm only depends on the covariance matrix $\Sigma$ and, therefore,  it is denoted by $\norm\cdot_{\text{HR}_{\Sigma}}$.
\end{itemize}

The generator of a $D$-norm is in general not uniquely determined, even its distribution is not. Take, for example, any rv $X>0$ with $E(X)=1$. Then $\bfZ=(Z_1,\dots,Z_d)=(X,\dots,X)$ generates the sup-norm $\norm\cdot_\infty$. An account of the theory of $D$-norms is provided by \citet{falk2019}.

\section{Generalized Pareto Copulas}\label{sec:generalized Pareto copulas}

Corollary \ref{coro:characterization of domain of attraction of copula} stimulates the following idea.  Choose an arbitrary $D$-norm $\norm\cdot_D$ on $\R^d$ and put with $\bfone=(1,\dots,1)\in\R^d$
\[
C(\bfu):= \max\left(1- \norm{\bfone-\bfu}_D,0\right),\qquad \bfu\in[0,1]^d.
\]
Each univariate margin $C_i$ of $C$, defined this way, satisfies for $u\in[0,1]$
\begin{align*}
C_i(u) &= C(1,\dots,1,\underbrace{u}_{\text{$i$-th component}},1\dots,1)\\
&= 1-\norm{(0,\dots,0,1-u,0,\dots,0)}_D\\
&= 1- (1-u)\underbrace{E(Z_i)}_{\text{$=1$}}= u,
\end{align*}
i.e., each $C_i$ is the uniform df on $(0,1)$. But $C$ does in general \emph{not} define a df, see, e.g., \citet[Proposition 5.1.3]{fahure10}. We require, therefore, the expansion
\[
C(\bfu)=1-\norm{\bfone-\bfu}_D
\]
only for $\bfu$ close to $\bfone\in\R^d$, i.e., for $\bfu\in[\bfu_0,\bfone]\subset\R^d$ with some $\bfzero<\bfu_0<\bfone\in\R^d$. A copula $C$ with this property will be called a \emph{generalized Pareto copula} (GPC). These copulas were introduced in \citet{aulbf11};  tests, whether data are generated by a copula in a $\delta$-neighborhood of a GPC were derived in \citet{aulff18}, see Section \ref{sec:delta-neighborhoods of GPC} for the precise definition of this neighborhood. The multivariate generalized Pareto distributions defined in Section \ref{sec:multivariate generalized Pareto distributions} show that GPC actually exist for any $D$-norm $\norm\cdot_D$. The corresponding construction of a generalized Pareto distributed rv also provides a way to simulate data from an arbitrary GPC.

As a consequence, an \emph{arbitrary} copula $C$ satisfies the following equivalences
\begin{align*}
&C\in\mathcal D(G)\\
&\iff C(\bfu)=1-\norm{\bfone-\bfu}_D + o(\norm{\bfone-\bfu})\quad\text{for some $D$-norm $\norm\cdot_D$}\\
&\iff C \text{\ is in its upper tail close to that of a GPC.}
\end{align*}
In this case we have $G(\bfx)=\exp(-\norm{\bfx}_D)$, $\bfx\le\bfzero\in\R^d$.
\begin{exam}
\upshape Take an arbitrary \emph{Archimedean copula}\index{Copula!Archimedean} on $\R^d$
\[
C_\varphi(\bfu)=\varphi^{-1}(\varphi(u_1)+\dots+\varphi(u_d)),
\]
where $\varphi$ is a continuous and strictly decreasing function from $(0,1]$ to $[0,\infty)$ such that $\varphi(1)=0$ (see, e.g., \citet[Theorem 2.2]{mcnn09}).
Suppose that
\begin{equation}\label{eqn:charpentier_segers_condition}
p:=-\lim_{s\downarrow 0}\frac{s\varphi\prime(1-s)}{\varphi(1-s)}\mbox{ exists in }[1,\infty).
\end{equation}
It follows from  \citet[Theorem 4.1]{chas09} that $C$ is in its upper tail close to the GPC with corresponding logistic $D$-norm $\norm\cdot_p$.

Suppose that the generator function $\varphi:(0,1]\to[0,\infty)$ satisfies with some $s_0\in (0,1)$
\begin{equation}\label{eqn:sufficient_property_of_Archimedean copula}
-\frac{s\varphi'(1-s)}{\varphi(1-s)} = p,\qquad s\in(0,s_0],
\end{equation}
with $p\in[1,\infty)$. Then $C_\varphi$ is a GPC, precisely,
\[
C_\varphi(\bfu)= 1-\norm{\bfone-\bfu}_p= 1- \left(\sum_{i=1}^d\abs{1-u_i}^p\right)^{1/p},\qquad \bfu\in[1-s_0,1]^d.
\]

This is readily seen as follows. Condition \eqref{eqn:sufficient_property_of_Archimedean copula} is equivalent with the equation
\[
(\log(\varphi(1-s)))' =\frac ps,\qquad s\in(0,s_0].
\]
Integrating both sides implies
\[
\log(\varphi(1-s)) - \log(\varphi(1-s_0)) = p\log(s)-p\log(s_0)
\]
or
\[\log\left(\frac{\varphi(1-s)}{\varphi(1-s_0)}\right)= \log\left(\left(\frac s{s_0} \right)^p \right),\qquad s\in(0,s_0],
\]
which implies
\[
\varphi(1-s)= \frac{\varphi(1-s_0)}{s_0^p}s^p,\qquad s\in[0,s_0],
\]
i.e.,
\[
\varphi(s)= c(1-s)^p, \qquad s\in[1-s_0,1],
\]
with $c:= \varphi(1-s_0)/s_0^p$. But this yields
\begin{align*}
C_\varphi(\bfu)&=\varphi^{-1}(\varphi(u_1)+\dots+\varphi(u_d))\\
&=1-\left(\sum_{i=1}^d (1-u_i)^p \right)^{1/p},\qquad \bfu\in[1-s_0,1]^d.
\end{align*}
\end{exam}

\section{Characterization of a GPC}

Next we derive the characteristic property of a GPC. Suppose the rv $\bfU$ follows a GPC $C$. Then its survival function equals
\[
P(\bfU\ge \bfu)=\dnormf{\bfone-\bfu}_D,\qquad \bfu\in[\bfu_0,\bfone]\subset\R^d,
\]
where
\[
\dnormf{\bfx}_D:= E\left(\min_{1\le i\le d}(\abs{x_i}Z_i) \right),\qquad \bfx\in\R^d,
\]
is the \emph{dual $D$-norm function} pertaining to $\norm\cdot_D$ with generator $\bfZ=(Z_1,\dots,Z_d)$, see the proof of Theorem \ref{propo:characteristic_property_of_GPC}. Using the equations \eqref{eqn:representation_of_maxima_via minima and vice versa} below it is straightforward to prove that $\dnormf{\cdot}_D$ does not depend on the particular choice of the generator $\bfZ$ of $\norm\cdot_D$.  We have, for example,
\[
\dnormf{\bfx}_1=0,\quad \dnormf{\bfx}_\infty=\min_{1\le i\le d}\abs{x_i},\qquad \bfx=(x_1,\dots,x_d)\in\R^d.
\]
Note that the mapping $\norm\cdot_D\mapsto \dnormf{\cdot}_D$ is not one-to-one, i.e., two different $D$-norms can have identical dual $D$-norm functions.

The function $\dnormf\cdot_D$ is obviously homogeneous:
\[
\dnormf{t\bfx}_D=t \dnormf{\, \bfx}_D,\qquad t\ge 0.
\]

As a consequence, a GPC is \emph{excursion stable:}
\[
P\left(\bfU\ge \bfone-t\bfu\mid \bfU\ge \bfone-\bfu\right) = \frac{\dnormf{t\bfu}_D}{\dnormf{\bfu}_D}=t,\quad t\in[0,1],
\]
for $\bfu$ close to $\bfzero\in\R^d$, provided $\dnormf{\bfu}_D>0$.

Note that each marginal distribution of a GPC $C$ is a lower dimensional GPC as well: If the rv $\bfU=(U_1,\dots,U_d)$ follows the GPC $C$ on $\R^d$, then the rv $\bfU_T:=(U_{i_1},\dots,U_{i_m})$ follows a GPC on $\R^m$, for each nonempty subset $T=\set{i_1,\dots,i_m}\subset\set{1,\dots,d}$. We have
\[
P\left(\left(U_{i_1},\dots,U_{i_m} \right)\le\bfv \right) = 1-\norm{\sum_{j=1}^m(1-v_j)\bfe_{i_j}}_D,
\]
for $\bfv=(v_1,\dots,v_m)\in[0,1]^m$ close to $\bfone\in\R^m$, where $\bfe_i=(0,\dots,0,1,0,\dots,0)\in\R^d$ denotes the $i$-th unit vector in $\R^d$, $1\le i\le d$.

The characteristic property of a GPC is its excursion stability\index{Excursion stability!of generalized Pareto copula}, as formulated in the next result.

\begin{theorem}\label{propo:characteristic_property_of_GPC}
Let the rv $\bfU=(U_1,\dots,U_d)$ follow a copula $C$. Then $C$ is a GPC  iff for each nonempty subset $T=\set{i_1,\dots,i_m}$ of $\set{1,\dots,d}$ the rv $\bfU_T=(U_{i_1},\dots,U_{i_m})$ is exceedance stable, i.e.,
\begin{equation}\label{eqn:exceedance stability of GPC}
P\left(\bfU_T\ge \bfone-t\bfu\right)=t P(\bfU_T\ge \bfone-\bfu),\qquad t\in [0,1],
\end{equation}
for $\bfu$ close to $\bfzero\in\R^m$.
\end{theorem}

\begin{proof}
The  implication ``$\Leftarrow$" in the preceding result is just a reformulation of \citet[Proposition 6]{fagu08}. The conclusion ``$\Rightarrow$" can be seen as follows. We can assume without loss of generality that $T=\set{1,\dots,d}$.

Using induction, it is easy to see that arbitrary numbers $a_1,\dots, a_d \in \R$ satisfy the equations
\begin{align}\label{eqn:representation_of_maxima_via minima and vice versa}
\max(a_1,\dots, a_d)= \sum_{\emptyset \neq T \subset \{1, \dots, d\}}(-1)^{\abs{T}-1}\min_{i \in T}\,a_i,\nonumber\\
\min(a_1,\dots, a_d)= \sum_{\emptyset \neq T \subset \{1, \dots, d\}}(-1)^{\abs{T}-1}\max_{i \in T}\,a_i.
\end{align}

By choosing $a_1=\dots=a_d=1$, the preceding equations imply in particular
\begin{equation}\label{eqn:inclusion_exclusion_cases=1}
1=  \sum_{\emptyset \neq T \subset \{1, \dots, d\}}(-1)^{\abs{T}-1}.
\end{equation}

 The inclusion-exclusion principle implies for $\bfv\in[0,1]^d$ close to $\bfzero\in\R^d$
\begin{align*}
P(\bfU\ge 1-\bfv)
&= 1-P\left(\bigcup_{i=1}^d\set{U_i\le 1-v_i}\right)\\
&= 1- \sum_{\emptyset\not=T\subset\set{1,\dots,d}} (-1)^{\abs T-1} P(U_i\le 1-v_i,\,i\in T)\\
&= 1- \sum_{\emptyset\not=T\subset\set{1,\dots,d}} (-1)^{\abs T-1}\left(1-\norm{\sum_{i\in T}v_i\bfe_i}_D\right)\\
&= \sum_{\emptyset\not=T\subset\set{1,\dots,d}} (-1)^{\abs T-1} \norm{\sum_{i\in T}v_i\bfe_i}_D.
\end{align*}

Choose a generator $\bfZ=(Z_1,\dots,Z_d)$ of $\norm\cdot_D$. From equation \eqref{eqn:representation_of_maxima_via minima and vice versa} we obtain
\begin{align*}
&\sum_{\emptyset\not=T\subset\set{1,\dots,d}} (-1)^{\abs T-1} \norm{\sum_{i\in T}v_i\bfe_i}_D\\
&= \sum_{\emptyset\not=T\subset\set{1,\dots,d}} (-1)^{\abs T-1} E\left(\max_{i\in T}(v_iZ_i) \right)\\
&= E\left(\sum_{\emptyset\not=T\subset\set{1,\dots,d}} (-1)^{\abs T-1} \max_{i\in T}(v_iZ_i) \right)\\
&= E\left(\min_{1\le i\le d}(v_iZ_i)\right)
= \dnormf{\bfv}_D.
\end{align*}
Replacing $\bfv$ by $t\bfu$  yields the assertion.
\end{proof}

If $P(\bfU_T\ge \bfone-\bfu)>0$, then \eqref{eqn:exceedance stability of GPC} clearly becomes
\[
P\left(\bfU_T\ge \bfone-t\bfu\mid \bfU_T\ge \bfone-\bfu\right)=t,\qquad t\in[0,1].
\]
But $P(\bfU_T\ge \bfone-\bfu)$ can be equal to zero for all $\bfu$ close to $\bfone\in\R^m$. This is for example the case, when the underlying $D$-norm $\norm\cdot_D$ is $\norm\cdot_1$. Then $\dnormf\cdot_D=0$, and, thus, $P(\bfU_T\ge \bfone-\bfu)=0$ for all $\bfu$ close to $\bfzero\in\R^m$, unless $m=1$.

While the characteristic property of a GPC is its excursion stability, the characteristic property of an extreme value copula $C_G(\bfu)=G\left(G_1^{-1}(u_1),\dots,G_d^{-1}(u_d) \right)$, $\bfu\in(0,1)^d$, which corresponds to a max-stable df $G$, is its max-stability, defined below. By transforming the univariate margins to the standard negative distribution, we can assume without loss of generality that $G$ is an SMS df. In this case we have $G_i^{-1}(u)=\log(u)$, $u\in(0,1]$, and, thus, we obtain the representation of the copula of an \emph{arbitrary} max-stable df
\begin{equation}\label{eqn:copula of SMS df}
C_G(\bfu)=\exp\left(-\norm{(\log(u_1),\dots,\log(u_d))}_D\right),\qquad \bfu\in(0,1]^d,
\end{equation}
with some $D$-norm $\norm\cdot_D$. For a discussion of parametric families of extreme value copulas and their statistical analysis we refer to \citet{gennes12}.

Equation \eqref{eqn:copula of SMS df} obviously implies the \emph{max-stability} of an extreme value copula $C_G$:
\begin{equation}\label{eqn:max-stability of extreme value copula}
C_G^n\left(\bfu^{1/n}\right) = C_G(\bfu),\qquad \bfu\in(0,1]^d,\;n\in\N.
\end{equation}
If, on the other hand, an \emph{arbitrary} copula $C$ satisfies equation \eqref{eqn:max-stability of extreme value copula}, then it is clearly the copula $C_G$ of a SMS df $G$.
As a consequence, we have two stabilities of copulas: max-stability and exceedance stability.

Let $C$ be an arbitrary copula on $\R^d$. The considerations in this section show that the copula $C_{C^n}$ of $C^n$ converges point-wise to a max-stable copula if, and only if, $C$ is in its upper tail close to that of an excursion stable copula, i.e., to that of a GPC.

The message of the considerations in this section is: If one wants to model the \emph{copula} of multivariate exceedances above high thresholds, then a GPC is a first option.

\section{Multivariate Generalized Pareto Distributions}\label{sec:multivariate generalized Pareto distributions}

Let $\set{G_\alpha:\,\alpha\in\R}$ be the set of univariate max-stable df as defined  by the equations above and  in \eqref{defn:of univariate EVD}.  The family of univariate \emph{generalized Pareto distributions} (GPD) is the family of univariate excursion stable distributions:
\begin{align*}
H_\alpha (x)
&:= 1+\log(G_\alpha(x)),\qquad G_\alpha(x)> \exp(-1),\\
&=\begin{cases}
1-(-x)^\alpha,\quad -1\le x\le 0,&\mbox{ if }\alpha>0,\\
1-x^\alpha,\qquad\quad\;\, x\ge 1,&\mbox{ if }\alpha<0,\\
1-\exp(-x),\quad x\ge 0,&\mbox{ if }\alpha=0.
\end{cases}
\end{align*}

Suppose the rv $V$ follows the df $H_\alpha$. Then
\begin{align*}
P(V>tx\mid V>x)&=t^\alpha\quad\text{for }\begin{cases}
t\in[0,1],& -1\le x<0,\text{ if }\alpha>0,\\
t\ge 1,& x\ge 1,\text{ if }\alpha<0,
\end{cases}\\
P(V>x+t\mid V>x)&=\exp(-t),\quad\text{for }t\ge 0,\quad x\ge 0,\text{ if }\alpha=0.
\end{align*}

For a threshold $s$ and an $x>s$, the univariate GPD takes the form of the following scale and shape family of distributions
\begin{equation}\label{eq:uni_GPD}
H_{1/\xi}((x-s)/\sigma)=1-\left(1+\xi(x-s)/\sigma\right)^{-1/\xi},
\end{equation}
where $\xi=1/\alpha$ and $\sigma>0$ \citep[e.g.][page 35]{fahure10}.

The definition of a \emph{multivariate} GPD is, however, not unique in the literature. There are different approaches (\citet{roott06}, \citet{fahure10}), each one trying to catch the excursion stability of a multivariate rv. The following suggestion might conclude this debate. Clearly, the excursion stability of a rv $\bfX$ should be satisfied by its margins \emph{and} its copula. This is reflected in the following definition.

\begin{defi}\label{def:definition of multivariate gpd}
\upshape A rv $\bfX=(X_1,\dots,X_d)$ follows a multivariate GPD, if each component $X_i$ follows a univariate GPD (at least in its upper tail), and if the copula $C$ corresponding to $\bfX$ is a GPC, i.e., there exists a $D$-norm $\norm\cdot_D$ on $\R^d$ and $\bfu_0\in[0,1)^d$ such that
\[
C(\bfu)=1-\norm{\bfone-\bfu}_D,\qquad \bfu\in[\bfu_0,\bfone].
\]
\end{defi}

As a consequence, each such rv $\bfX$, which follows a multivariate GPD, is exceedance stable and vice versa.

\begin{exam}
\upshape The following construction extends the bivariate approach proposed by \citet{buihz08} to arbitrary dimension. It provides a rv, which follows an arbitrary multivariate GPD as in Definition \ref{def:definition of multivariate gpd}. Let $\bfZ=(Z_1,\dots,Z_d)$ be the generator of a $D$-norm $\norm\cdot_D$, with the additional property that each $Z_i\le c$, for some $c\ge 1$. Note that such a generator exists for an arbitrary $D$-norm according to the \emph{normed generators theorem} for $D$-norms (\citet{falk2019}). Let the rv $U$ be uniformly on $(0,1)$ distributed and independent of $\bfZ$. Put
\begin{equation}\label{exam:simple multivariate GPD rv}
\bfV=(V_1,\dots,V_d):= \frac 1U(Z_1,\dots,Z_d):= \frac 1U \bfZ.
\end{equation}
Then, for each $i\in\set{1,\dots,d}$,
\[
P\left(\frac 1U Z_i \le x \right) = 1- \frac 1x, \qquad x \text{ large},
\]
i.e., $V_i$ follows in its upper tail a univariate standard Pareto distribution, and, by elementary computation, we have
\[
P\left(\bfV\le \bfx \right) = 1- \norm{\frac{\bfone}{\bfx}}_D,\qquad \bfx\text{ large}.
\]
The preceding equation implies that the copula of $\bfV$ is a GPC with corresponding $D$-norm $\norm\cdot_D$. The rv $\bfV$ can be seen as a prototype of a rv, which follows a multivariate GPD. This GPD is commonly called \emph{simple}.

Choose $\bfV=(V_1,\dots,V_d)$ as in equation \eqref{exam:simple multivariate GPD rv} and numbers $\alpha_1,\dots,\alpha_d\in\R$. Then
\begin{align}\label{eqn:prototype representation of general multivariate gpd}
\bfY&:=(Y_1,\dots,Y_d)\nonumber\\
&:= \left(H_{\alpha_1}^{-1}\left(1-\frac 1{V_1}\right),\dots, H_{\alpha_d}^{-1}\left(1-\frac 1{V_d}\right)\right)\nonumber\\
&= \left(H_{\alpha_1}^{-1}\left(1-\frac U{Z_1}\right),\dots, H_{\alpha_d}^{-1}\left(1-\frac U{Z_d}\right)\right)
\end{align}
follows a \emph{general} multivariate GPD with margins $H_{\alpha_1},\dots,H_{\alpha_d}$ in its univariate upper tails.

With the particular choice $\alpha_1=\dots=\alpha_d=1$ we obtain a \emph{standard} multivariate GPD
\[
\bfY = -U\left(\frac 1{Z_1},\dots,\frac1{Z_d}\right).
\]
Its df is
\[
P(\bfY\le\bfx)=1-\norm{\bfx}_D
\]
for $\bfx\le\bfzero\in\R^d$, close enough to zero.

With the particular choice $\alpha_1=\dots=\alpha_d=0$ we obtain a multivariate GPD with Gumbel margins in the upper tails
\[
\bfY= \left(\log(Z_1)-\log(U),\dots,\log(Z_d)-\log(U)\right),
\]
where $-\log(U)$ follows the standard exponential distribution on $(0,\infty)$.

Up to a possible location and scale shift, \emph{each} rv $\bfX=(X_1,\dots,X_d)$, which follows a multivariate GPD as defined in Definition \ref{def:definition of multivariate gpd}, can in its upper tail be modeled by the rv $\bfY=(Y_1,\dots,Y_d)$ in equation \eqref{eqn:prototype representation of general multivariate gpd}. This makes such rv $\bfY$ in particular natural candidates for simulations of multivariate exceedances above high thresholds.
\end{exam}

\section{$\delta$-Neighborhoods of GPC}\label{sec:delta-neighborhoods of GPC}

A major problem with the construction in \eqref{exam:simple multivariate GPD rv} is the additional boundedness condition on the generator $\bfZ$. This is, for example, not given in case of the logistic $D$-norm $\norm\cdot_p$ with $p\in(1,\infty)$ or the H\"{u}sler-Reiss $D$-norm. From the normed generators theorem in \citet{falk2019} we know that bounded generators exist, but, to the best of our knowledge, they are unknown in both cases.

In this section we drop this boundedness condition and show that  the construction \eqref{exam:simple multivariate GPD rv} provides a copula, which is in a particular neighborhood of a GPC, called $\delta$-neighborhood. We are going to define this neighborhood next.

 Denote by $R:=\set{\bft\in[0,1]^d:\,\norm{\bft}_1=1}$ the unit sphere in $[0,\infty)^d$ with respect to the norm $\norm{\bfx}_1=\sum_{i=1}^d\abs{x_i}=1$, $\bfx\in\R^d$. Choose an arbitrary copula $C$ on $\R^d$ and put for $\bft\in R$
 \[
 C_{\bft}:= C(\bfone+s\bft),\qquad s\le 0.
 \]
 Then $C_{\bft}$ is a univariate df on $(-\infty,0]$, and the copula $C$ is obviously determined by the family
 \[
 \mathcal P(C):=\set{C_{\bft}:\,\bft\in R}
 \]
 of univariate \emph{spectral} df $C_{\bft}$. The family $\mathcal P(C)$ is the \emph{spectral decomposition} of $C$; cf \citet[Section 5.4]{fahure10}. A copula $C$ is, consequently, in $\mathcal D(G)$ iff its spectral decomposition satisfies
 \[
 C_{\bft}(s)=1+s\norm{\bft}_D + o(s),\qquad \bft\in R,
 \]
 as $s\uparrow 0$. The copula $C$ is by definition in the \emph{$\delta$-neighborhood} of the GPC $C_D$ with $D$-norm $\norm\cdot_D$ if their upper tails are close to one another, precisely, if
 \begin{align}\label{defn:of delta-neighborhood}
 1-C_{\bft}(s)&= (1-C_{D,\bft}(s))\left(1+O\left(\abs s^\delta\right)\right)\nonumber\\
  &= \abs s\norm{\bft}_D \left(1+O\left(\abs s^\delta\right)\right)
 \end{align}
 as $s\uparrow 0$, uniformly for $\bft\in R$. In this case we know from \citet[Theorem 5.5.5]{fahure10} that
 \begin{equation}\label{eqn:polynomial convergence of copula}
 \sup_{\bfx\in(-\infty,0]^d}\abs{C^n\left(\bfone+\frac 1n \bfx\right)-\exp(-\norm{\bfx}_D)}=O\left(n^{-\delta}\right).
 \end{equation}

 Under additional differentiability conditions on $C_{\bft}(s)$ with respect to $s$, also the reverse implication $\eqref{eqn:polynomial convergence of copula}\implies\eqref{defn:of delta-neighborhood}$ holds; cf.  \citet[Theorem 5.5.5]{fahure10}. Therefore, the $\delta$-neighborhood of a GPC, roughly, collects those copula with a polynomial rate of convergence for maxima.

 Condition \eqref{defn:of delta-neighborhood} can also be formulated in the following way:
 \begin{align*}
  1-C(\bfu) &= (1-C_D(\bfu))\left(1+O\left(\norm{\bfone-\bfu}^\delta\right) \right)\\
  &= \norm{\bfone-\bfu}_D\left(1+O\left(\norm{\bfone-\bfu}^\delta\right)\right)
 \end{align*}
as $\bfu\to\bfone\in\R^d$, uniformly for $\bfu\in[0,1]^d$, where $\norm\cdot$ is an arbitrary norm on $\R^d$.

 \begin{exam}
 \upshape Choose $\bfu\in (0,1)^d$ and put for $t\in[0,1]$
 \[
 \text{FI}\,(t,\bfu):= E\left(\sum_{i=1}^d 1_{(U_i>1-tu_i)}\,\Big|\,\sum_{i=1}^d 1_{(U_i>1-u_i)} >0 \right).
 \]
 With $t=1$, this is the \emph{fragility index}, introduced by \citet{gelhv07} to measure the stability of the stochastic system $U_1,\dots,U_d$. The system is called \emph{stable} if $\text{FI}\,(1,\bfu)$ is close to one, otherwise it is called \emph{fragile}. The asymptotic distribution of $N_{\bfu}=\sum_{i=1}^d 1_{(U_i>1-tu_i)}$, given $N_{\bfu}>0$, was investigated in \citet{falt11, falt10a}.

 If $\bfU$ follows a GPC with corresponding $D$-norm $\norm\cdot_D$, we obtain for $\bfu$ close enough to zero
 \begin{align*}
 \text{FI}\,(t,\bfu) &= \sum_{i=1}^d \frac{P(U_i>1-tu_i)}{P\left(\sum_{j=1}^d 1_{(U_j>1-u_j)}>0\right)}\\
 &= \sum_{i=1}^d \frac{tu_i}{1-P(\bfU\le 1-\bfu)}\\
 &= t \frac{\norm{\bfu}_1}{\norm{\bfu}_D}.
 \end{align*}

 Writing
 \[
 \frac{\norm{\bfu}_1}{\norm{\bfu}_D} = \frac 1{\norm{\frac{\bfu}{\norm{\bfu}_1}}_D}
 \]
 implies that there is a least favorable direction $\bfr_0\in R$ with
 \[
 \norm{\bfr_0}_D = \min_{\bfr\in R}\norm{\bfr}_D.
 \]
 A vector $\bfu$ with $\bfu=s\bfr_0$, $s>0$, maximizes the fragility index.  For arbitrary $d\ge 2$ and $\norm\cdot_D=\norm\cdot_p$, $p\in(1,\infty)$, one obtains for example $\bfr_0$ with constant entry $1/d$ and
 \[
\text{FI}\,(t,\bfu)= t \frac d{d^{1/p}}.
 \]

 If $\bfU$ follows a copula, which is in a $\delta$-neighborhood of a GPC with $D$-norm $\norm\cdot_D$, then we obtain the representation
 \[
 \text{FI}\,(t,\bfu) = t \frac{\norm{\bfu_1}}{\norm{\bfu}_D}\left(1+O\left(\norm{\bfu}^\delta\right)\right),\qquad \text{for }\bfu\to\bfzero\in\R^d.
 \]

If we replace $\bfU$ for example by $\bfX=\left(F^{-1}(U_1),\dots,F^{-1}(U_d)\right)$, where $F(x)=1-1/x$, $x\ge 1$, is the standard Pareto df, then we obtain for the fragility index
\[
\text{FI}\,(t,\bfx) =E\left(\sum_{i=1}^d 1_{(X_i>tx_i)}\,\Big|\, \sum_{i=1}^d 1_{(X_i>x_i)}>0 \right),\qquad \bfx\ge\bfone\in\R^d,\;t\ge 1,
\]
the equality
\[
\text{FI}\,(t,\bfx) = \frac 1t \frac{\norm{\bfone/\bfx}_1}{\norm{\bfone/\bfx}_D}\left(1+O\left(\norm{\bfone/\bfx}^\delta\right)\right)\qquad \text{for }x_i\to\infty,\,1\le i\le d.
\]

 \end{exam}

 Let $\bfZ=(Z_1,\dots,Z_d)$ be a generator of the $D$-norm $\norm\cdot_D$ and let $U$ be a rv, which is independent of $\bfZ$ and which follows the uniform distribution on $(0,1)$. If $\bfZ$ is bounded, then the copula of $\bfZ/U$ is a GPC $C_D$ as established in Section \ref{sec:multivariate generalized Pareto distributions}. If we drop the boundedness of $\bfZ$ and require that $E(Z_i^2)<\infty$, then, roughly, the copula of $\bfZ/U$ is in a $\delta$-neighborhood of $C_D$ with $\delta=1$. This is the content of our next result.

 \begin{theorem}\label{theo:copula of Z/U is in delta-neighborhood}
 Let $\bfZ=(Z_1,\dots,Z_d)$ generate the $D$-norm $\norm\cdot_D$. Suppose that $E(Z_i^2)<\infty$ and that the df of $Z_i$ is continuous, $1\le i\le d$. Then the copula  $C_{\bfV}$ of
 \[
 \bfV:= \frac 1U\bfZ=\frac 1U(Z_1,\dots,Z_d)
 \]
 is in the $\delta$-neighborhood of the GPC $C_D$ with $\delta=1$.
 \end{theorem}

 \begin{proof}
 The df $F_i$ of $Z_i/U$ satisfies for large $x$
 \begin{align*}
 F_i(x)&=P(Z_i/x\le U)\\
 &=\int_0^x P(U\ge z/x)\,(P*Z_i)(dz)\\
 &= \int_0^x 1-\frac zx \,(P*Z_i)(dz)\\
 &= P(Z_i\le x) - \frac 1x E\left(Z_i1_{(Z_i\le x)}\right)\\
 &= 1- P(Z_i>x) - \frac 1x \left(1-E\left(Z_i 1_{(Z_i>x)}\right)\right)\\
 &= 1-\frac 1x - \left(P(Z_i>x) - \frac 1x E\left(Z_i 1_{(Z_i>x)}\right)  \right)\\
 &= \left(1-\frac 1x\right)\left(1- \frac{P(Z_i>x) - \frac 1x E\left(Z_i 1_{(Z_i>x)}\right)}{1-\frac 1x} \right),
 \end{align*}
 where by Tschebyscheff's inequality
 \[
 P(Z_i>x)\le \frac 1{x^2} E(Z_i^2)
 \]
 and, using also H\"{o}lder's inequality
 \[
  E\left(Z_i 1_{(Z_i>x)}\right)\le E(Z_i^2)^{1/2}P(Z_i>x)^{1/2}\le E(Z_i^2)^{1/2} \frac{E(Z_i^2)^{1/2}}x= \frac 1x E(Z_i^2).
 \]
 As a consequence we obtain
 \[
 F_i(x) =\left(1-\frac 1x\right)\left(1+O\left( \frac 1{x^2} \right)\right) \qquad \text{as }x\to\infty
 \]
 and, thus,
 \[
 1-F_i(x)=\frac 1 x \left(1+O\left(\frac 1x\right)\right) \qquad \text{as }x\to\infty.
 \]
 Therefore, the df $F_i$ of $Z_i/U$ is in the $\delta$-neighborhood of the standard Pareto distribution with $\delta=1$.

 From \citet[Proposition 2.2.1]{fahure10} we obtain as a consequence
 \[
 F_i^{-1}(1-q) = \frac 1q (1+O(q))
 \]
 for $q\in (0,1)$ as $q\to 0$.

 Note that each df $F_i$ is continuous, $1\le i\le d$. Choose $\bft=(t_1,\dots,t_d)\in R$. We have for $s<0$ close enough to zero
 \begin{align*}
 &C_{\bft}(s)\\
 &= P(F_i(Z_i/U)\le 1+st_i,\,1\le i\le d)\\
 &= P(Z_i/U\le F_i^{-1}(1+st_i),\,1\le i\le d)\\
 &= P\left(\frac{Z_i}U\le \frac 1{\abs st_i}(1+O(s)),\,1\le i\le d \right)\\
 &= P(U\ge \abs s t_i(1+O(s))Z_i, \,1\le i\le d)\\
 &= P\left(U\ge \abs s \max_{1\le i\le d}(t_i(1+O(s))Z_i) \right)\\
 &=\int_{\set{\max_{1\le i\le d}(t_i(1+O(s))z_i)\le 1/\abs s }} P\left(U\ge \abs s \max_{1\le i\le d}(t_i(1+O(s))z_i) \right)\,(P*\bfZ)(d\bfz)\\
 &=\int_{\set{\max_{1\le i\le d}(t_i(1+O(s))z_i)\le 1/\abs s }} 1-\abs s \max_{1\le i\le d}(t_i(1+O(s))z_i) \,(P*\bfZ)(d\bfz)\\
 &= P\left(\max_{1\le i\le d}(t_i(1+O(s))Z_i)\le \frac 1{\abs s}  \right)\\
 &\hspace*{1cm}-\abs s E\left(\max_{1\le i\le d}(t_i(1+O(s))Z_i)\; 1_{\left(\max_{1\le i\le d}(t_i(1+O(s))Z_i)\le\frac 1{\abs s}\right)} \right)\\
 &= 1- P\left(\max_{1\le i\le d}(t_i(1+O(s))Z_i)> \frac 1{\abs s}  \right)
-\abs s E\left(\max_{1\le i\le d}(t_i(1+O(s))Z_i)\right)\\
&\hspace*{1cm}+\abs s E\left(\max_{1\le i\le d}(t_i(1+O(s))Z_i)\; 1_{\left(\max_{1\le i\le d}(t_i(1+O(s))Z_i)>\frac 1{\abs s}\right)} \right).
 \end{align*}

 We have
 \[
 E\left(\max_{1\le i\le d}(t_i(1+O(s))Z_i)\right)=E\left(\max_{1\le i\le d}(t_iZ_i)\right) (1+O(s)) = \norm{\bft}_D (1+O(s))
 \]
 and, thus, applying Tschebyscheff's inequality and H\"{o}lder's inequality again,
 \begin{align*}
 &1-C_{\bft}(s)\\
 &= P\left(\max_{1\le i\le d}(t_i(1+O(s))Z_i)>\frac 1{\abs s} \right) + \abs s\norm{\bft}_D (1+O(s))\\
 &\hspace*{1cm}- \abs s E\left(\max_{1\le i\le d}(t_i(1+O(s))Z_i)\; 1_{\left(\max_{1\le i\le d}(t_i(1+O(s))Z_i)>\frac 1{\abs s}\right)} \right)\\
 &= \abs s \norm{\bft}_D (1+O(s))\\
 &= \left(1-C_{D,\bft}(s)\right)  (1+O(s))
 \end{align*}
 as $s\uparrow 0$, uniformly for $\bft\in R$. Note that there exist constants $K_1,K_2>0$ such that
 $K_1\le \norm{\bft}_D\le K_2$ for each $\bft\in R$. This completes the proof of  Theorem \ref{theo:copula of Z/U is in delta-neighborhood}.
 \end{proof}

An obvious example is the generator of a H\"{u}sler-Reiss $D$-norm
\[
\bfZ^{(1)}=\left(\exp\left(X_1-\frac{\sigma_{11}}2\right),\dots, \exp\left(X_d-\frac{\sigma_{dd}}2\right)  \right),
\]
where $\bfX=(X_1,\dots,X_d)$ is multivariate normal $N(\bfzero,\Sigma)$, $\Sigma=(\sigma_{ij})_{1\le i,j\le d}$.

Another example is the generator of the logistic norm $\norm\cdot_p$, $p\in(2,\infty)$,
\[
\bfZ^{(2)} = (Y_1,\dots,Y_d)/\Gamma(1-1/p),
\]
where $Y_1,\dots,Y_p$ are iid Fr\'{e}chet distributed with df $F(x)=\exp(x^{-p})$, $x>0$, with parameter $p>2$.

Both generators are unbounded, but they have square integrable components with continuous df. It is known that \emph{bounded} generators actually exist in both cases, but to the best of our knowledge, they are unknown.

 \citet{aulff18} propose and extensively discuss a $\chi^2$-goodness-of-fit test for testing, whether the underlying copula of iid rv in arbitrary dimension is in the $\delta$-neighborhood of a GPC with an arbitrary $\delta>0$. This test might also used to test for a GPC.
 
\section{Estimation of Exceedance Probability}\label{sec:estimaton of exceedance probability}

In this section we apply the preceding results to derive estimates of the probability that a rv $\bfU=(U_1,\dots,U_d)$, which follows a copula, realizes in an interval $[\bfx_0,\bfone]\subset[0,1]^d$, where $\bfx_0$ is close to $\bfone\in\R^d$ and, thus, there are typically no observations available to estimate this probability by its empirical counterpart. This is a typical applied problem in extreme value analysis.

Suppose that the copula of $\bfU$, say $C$, is in the domain of attraction of a max-stable df.  In this case, its upper tail is by Corollary \ref{coro:characterization of domain of attraction of copula} close to that of a GPC.

We assume that the  copula $C$ is a GPC (or very close to one in its upper tail). Being a GPC is by Theorem \ref{propo:characteristic_property_of_GPC} characterized by the equation
\begin{equation}\label{eqn:characteristic property of GPC}
P(\bfU\ge \bfone-t\bfu) = tP(\bfU\ge \bfone-\bfu),
\end{equation}
$t\in[0,1],$ for $\bfu\ge \bfzero\in\R^d$ close enough to zero.

We want to estimate
\begin{equation}\label{eq:exc_prob_unif}
q:= P(\bfU\ge \bfx_0)
\end{equation}
for some $\bfx_0$ close to one, based on independent copies $\bfU^{(1)},\dots,\bfU^{(n)}$ of $\bfU$. Even more, we want to derive confidence interval pertaining to our estimators of $q$.

Choose $\bfu_0$ close to zero, such that equation \eqref{eqn:characteristic property of GPC} is satisfied for each $t\in[0,1]$, and
\begin{equation}\label{eq:x0_u0_relation}
\bfx_0 =1-t_0\bfu_0
\end{equation}
with some $t_0\in (0,1)$. Then the unknown probability $q$ satisfies the equation
\begin{equation}\label{eqn:crucial equation for unknown probability}
q=P(\bfU\ge \bfone-t_0\bfu_0)=t_0 P(\bfU\ge \bfone- \bfu_0)=:t_0 p.
\end{equation}

The threshold $\bfone- \bfu_0$ should be much smaller than the initial threshold $\bfx_0=\bfone-t_0\bfu_0$, in which case the the unknown probability $p$ can be estimated from the data by
\[
\hat p_n:=\frac 1n \sum_{i=1}^n 1\left(\bfU^{(i)}\ge \bfone-\bfu_0\right).
\]

Note that $n\hat p_n$ is binomial distributed $B(n,p)$; a confidence interval for $p$ can be obtained by Clopper-Pearson, for example. A popular approach is due to \citet{agresticoull}; see also \citet{browncaidasgupta}.

A confidence interval for $p$, say $I=(a,b)$, can by equation \eqref{eqn:crucial equation for unknown probability} be turned into a confidence interval $I^*$ for $q$ (with the same confidence level) by putting
\[
I^*:= t_0 I=(t_0a,t_0b).
\]

\subsection{Determination of $u_0$}

It is clear that one would like to choose $\bfu_0$ as large as possible, so that one has more observations in $[\bfone-\bfu_0,\bfone]$. But, on the other hand, the GPC property equation \eqref{eqn:characteristic property of GPC} needs to be satisfied as well. In what follows we describe a proper way how to choose $\bfu_0$.

A possible solution to check, if condition \eqref{eqn:characteristic property of GPC} is satisfied for $\bfu_0=(u_{01},\dots,u_{0d})$, is as follows. If condition \eqref{eqn:characteristic property of GPC} is satisfied, then we obtain for the conditional distribution
\[
P(\bfU\ge \bfone-t\bfu_0\mid \bfU\ge \bfone-\bfu_0)=t,\qquad t\in[0,1],
\]
or
\begin{align*}
P\left(\max_{1\le j\le d}\left(\frac{1-U_j}{u_{0j}}\right)\le t\,\big|\, \max_{1\le j\le d}\left(\frac{1-U_j}{u_{0j}}\right)\le 1  \right)= t,\qquad t\in[0,1].
\end{align*}

\

This means that those observations in the data $\max_{1\le j\le d}\left(\left(1-U_j^{(i)}\right)/u_{0j}\right)$, $1\le i\le n$, which are not greater than one,
actually follow the uniform distribution on $(0,1)$. We denote these by $M_1,\dots,M_m$, where their number $m$ is a random variable:
\[
m=\sum_{i=1}^n 1\left(\max_{1\le j\le d}\left(\left(1-U_j^{(i)}\right)/u_{0j}\right)\le 1\right).
\]
It is easy to check, if $M_1,\dots,M_m$ are independent and on $(0,1)$  uniformly distributed random variables, conditional on $m$. Standard goodness-of-fit tests like the Kolmogorov-Smirnov test or the Cram\'{e}r-von Mises test can be applied. Alternatively, $M_1,\dots,M_m$ can be transformed to independent standard normal random variables by considering $\Phi^{-1}(M_i)$, and standard tests for normality such as the Shapiro-Wilk test can be applied. The preceding problem was already discussed in \citet[Section 5.8]{fahure10}.

Put for $t\in[0,1]$
\[
\bfu(t):= \frac {\bfone-\bfx_0}t.
\]
Then, clearly,
\[
\bfx_0=(x_{01},\dots,x_{0d})=\bfone-t\bfu(t),\qquad t\in[0,1].
\]
But as $\bfu(t)$ needs to be in $[0,1]^d$, we obtain the restriction
\[
0\le \frac{1-x_{0j}}t\le 1,\qquad 1\le j\le d,
\]
or
\[
1-x_{0j}\le t\le 1,\qquad 1\le j\le d,
\]
i.e.,
\[
t_{\text{low}}:=\max_{1\le j\le d}(1-x_{0j})\le t\le 1.
\]
Choosing $\bfu_0$ as large as possible now becomes choosing $t\ge t_{\text{low}}$ as small as possible.

Put for $t\in[t_{\text{low}},1]$
\[
\hat q_n(t):= t\hat p_n(t):= \frac tn \sum_{i=1}^n1\left(\bfU^{(i)}\ge \bfone-\bfu(t)\right).
\]

We obtain for each $t\in[t_{\text{low}},1]$ observations $M_1(t),\dots,M_{m(t)}(t)$ in the data $\max_{1\le j\le d}\left(\left(1-U_j^{(i)}\right)/u_j(t)\right)$, $1\le i\le n$, which are not greater than one. We check for each $t$, whether $M_1(t),\dots,M_{m(t)}(t)$ follow the uniform distribution on $(0,1)$ by plotting corresponding $p$-value functions:
\[
(t,p_1(t)),\quad (t,p_2(t)), \qquad t\in[t_l,1].
\]
Precisely, we plot the minimum of $p_1(t)$ and $p_2(t)$, obtained from the Kolmogorov-Smirnov test and the Cramer-Von Mises test.

A candidate for $t_0$ is the lowest possible value that leads to a minimum $p$-value of at least 50\%. This is done in Figure \ref{fig:plot of p-value functions}.

\begin{figure}[ht]
\includegraphics[width=0.9\textwidth]{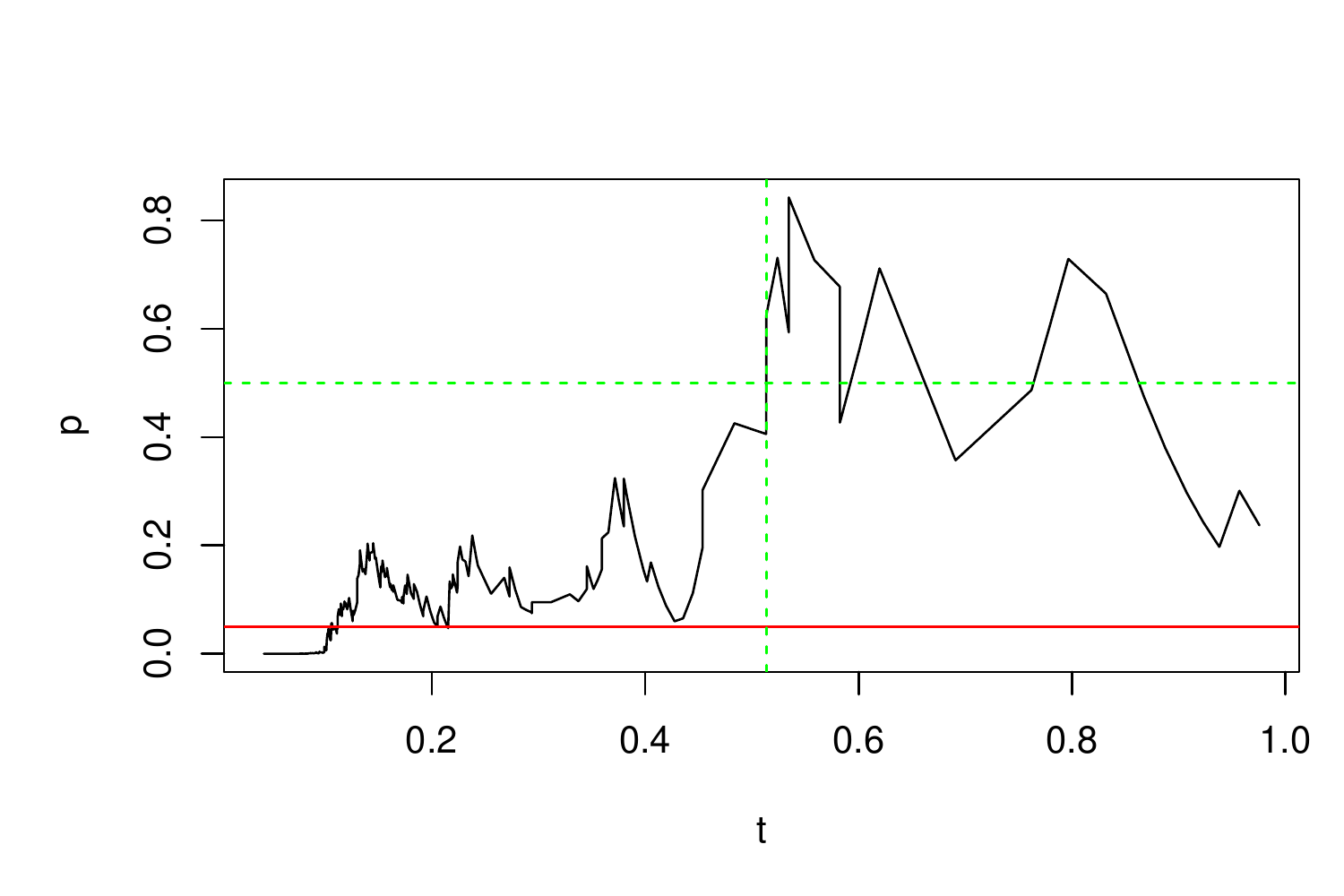}
\caption{\ Plot of $p$-value functions for testing for the uniform distribution on $(0,1)$.}
\label{fig:plot of p-value functions}
\end{figure}

\subsection{Confidence interval}

Now that we have chosen $\bfu_0$, we can estimate $p=P(\bfU\ge \bfu_0)$ as described before by
\[
\hat p_n:=\frac 1n \sum_{i=1}^n 1\left(\bfU^{(i)}\ge \bfone-\bfu_0\right).
\]

Under our model assumptions, the random variable $n\hat p_n$ is binomial distributed $B(n,p)$ and a confidence interval for $p$ can be obtained by Clopper-Pearson or the \citet{agresticoull} approach.

Figure \ref{fig:hat q_n(t)} shows $\hat q_n(t)$ together with the upper and lower limits of the corresponding confidence interval at the $95\%$ level. The green line marks the selected value for $t_0$ from before. Both plots are derived from scenario 6 in the following case study.

\begin{figure}[t!]
\includegraphics[width=0.9\textwidth]{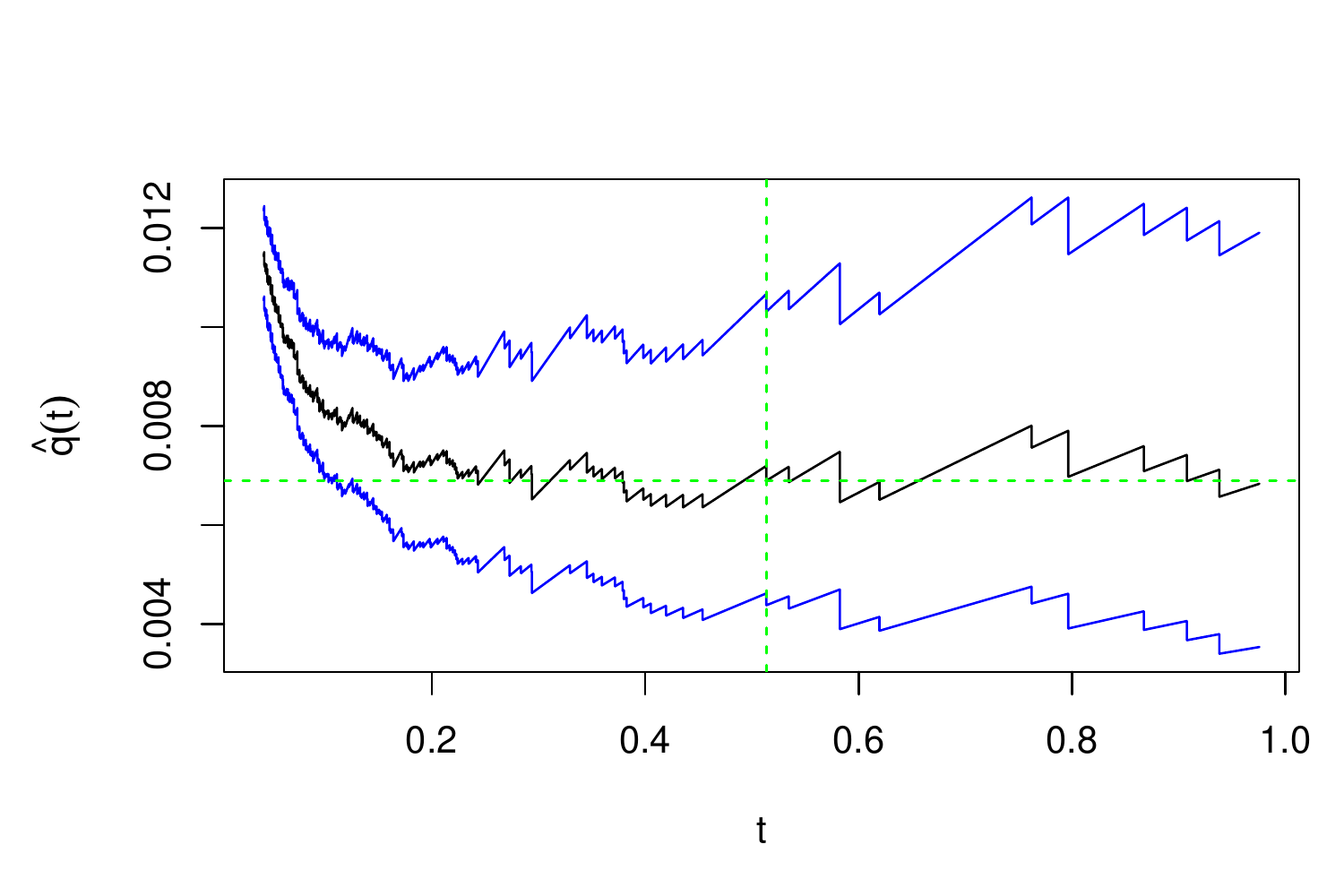}
\caption{\ Plot of the function $\hat q_n(t):= t\hat p_n(t)$ with confidence limits.}
\label{fig:hat q_n(t)}
\end{figure}

\section{A Case Study}\label{sec:case study}

Air pollution is an important social issue.
It is well-recognized that high emissions of air pollutants have a negative impact on the environment, climate and living being,
e.g. \citet[and the reference therein]{rossi1999air, brunekreef2002air, world2006air, guerreiro2016air}.
According to \citet{guerreiro2016air}, over a number of decades
 the European policy on the air-quality standards have assisted in reducing emissions of air pollutants.
The European air pollution directives regulate emissions of certain pollutants as ozone (O$_3$), nitrogen dioxide (NO$_2$), nitrogen oxide (NO), sulphur dioxide (SO$_2$)  and particulate matter (PM$_{10}$), with the aim of reducing the risk of negative effects on human health and environment that these might cause. The last three pollutants are mainly
produced by fuel motor vehicles, industry and house-heating, while the first two are produced by some reactions in the atmosphere.
On the basis of the World Health Organization (WHO) guidelines (\citet{world2006air}), the European emission regulation for air quality standard provides some pollutants concentrations that should not be exceeded.
\begin{table}[t]
\begin{center}
 \begin{tabular}{ccccc}
\toprule
Pollutant & Threshold & Period & Value in $\mu$g/m$^3$ & Recommendation\\
\midrule
O$_3$ & Limit & Daily max & 120 &  no more than 25\\
&&&& exceedances per year\\
& Information &  & 180 & \\
 & Alert &  & 240 & \\
 \midrule
NO$_2$ & Limit & 1-hour mean & 200 & no more than 18\\
&&&& exceedances per year\\
 & Alert &  & 400 & \\
 \midrule
%
SO$_2$ & Limit & 24-hour mean & 125 & no more than 3\\
&&&& exceedances per year\\
\midrule
PM$_{10}$ & Limit & 24-hour mean & 50 & no more than 35 \\
&&&& exceedances per year\\
& Target &  & 150 & \\
\bottomrule
\end{tabular}
\caption{Pollutant concentrations (thresholds) that should not be exceeded according the European emission regulation for air quality standard.}\label{tab:thresholds}
\end{center}
\end{table}
Table \ref{tab:thresholds} reports the short-term guideline values (see \citet[Chapters 10--13]{world2006air}, \citet[Chapters 4--6,8]{guerreiro2016air}). For NO
the same thresholds than those for NO$_2$ can be considered.
Meeting the short-term concentrations protects against air pollution peaks which can be dangerous to health.
The Limit threshold is a high percentile of the pollutant concentration (e.g. hourly, daily mean) in a year. It is recommended not to exceed this threshold with the objective to minimize health effects. Similarly, Target thresholds are proposed for reduction of air pollution when the pollutant concentrations are still considered very high.
Finally, in a country when the Information threshold is exceeded the
authorities need to notify their citizens by a public information notice. While, when the Alert threshold is exceeded for three consecutive hours, the authorities need to draw up a shortâterm action plan in accordance with specific provisions established in European Directive.
The threshold values are set for each individual pollutant without taking into account the dependence among pollutants. However,  it is well understood that certain pollutants can be dependent on each other; see, e.g., \citet{dahlhaus2000graphical, clapp2001analysis, heffernan2004conditional, world2006air}.

Here, we investigate which combinations of thresholds in Table \ref{tab:thresholds} are likely to be jointly exceeded and which ones are not.
Exceedances of individual thresholds are scarce when these are indeed high pollutant concentrations. This implies in this case that joint exceedances are even more rare. The latter event is a very rare one but it is a very severe pollution episode. Therefore, accurate estimation of joint exceedance probabilities is an important task. We show how to perform this ambitious mission using the method described in the previous section.
We do so analyzing the concentration of O$_3$, NO$_2$, NO, SO$_2$ and PM$_{10}$,  measured at the ground level in $\mu$g/m$^3$ in the Milan city center, Italy, during the years 2002--2017. Data are collected and made available by the Italian government agency  \emph{Agenzia Regionale per la Protezione dell'Ambiente (ARPA)}, see
{\tt http://www.arpalombardia.it/sites/QAria}. The first four pollutants are recorded in the average hourly format while the fifth in the daily average. To reveal the dependence among the pollutants we focus on two seasons: summer (May--August) and winter (November--February) (\citet{heffernan2004conditional}).
Since the thresholds in Table  \ref{tab:thresholds} are designed for different averaging periods, for comparison purposes we  focus on the daily maximum (of hourly averages) for all the pollutants except for PM$_{10}$ where we are forced to consider the daily average.
\begin{figure}[t!]
\includegraphics[width=0.6\textwidth, page=1]{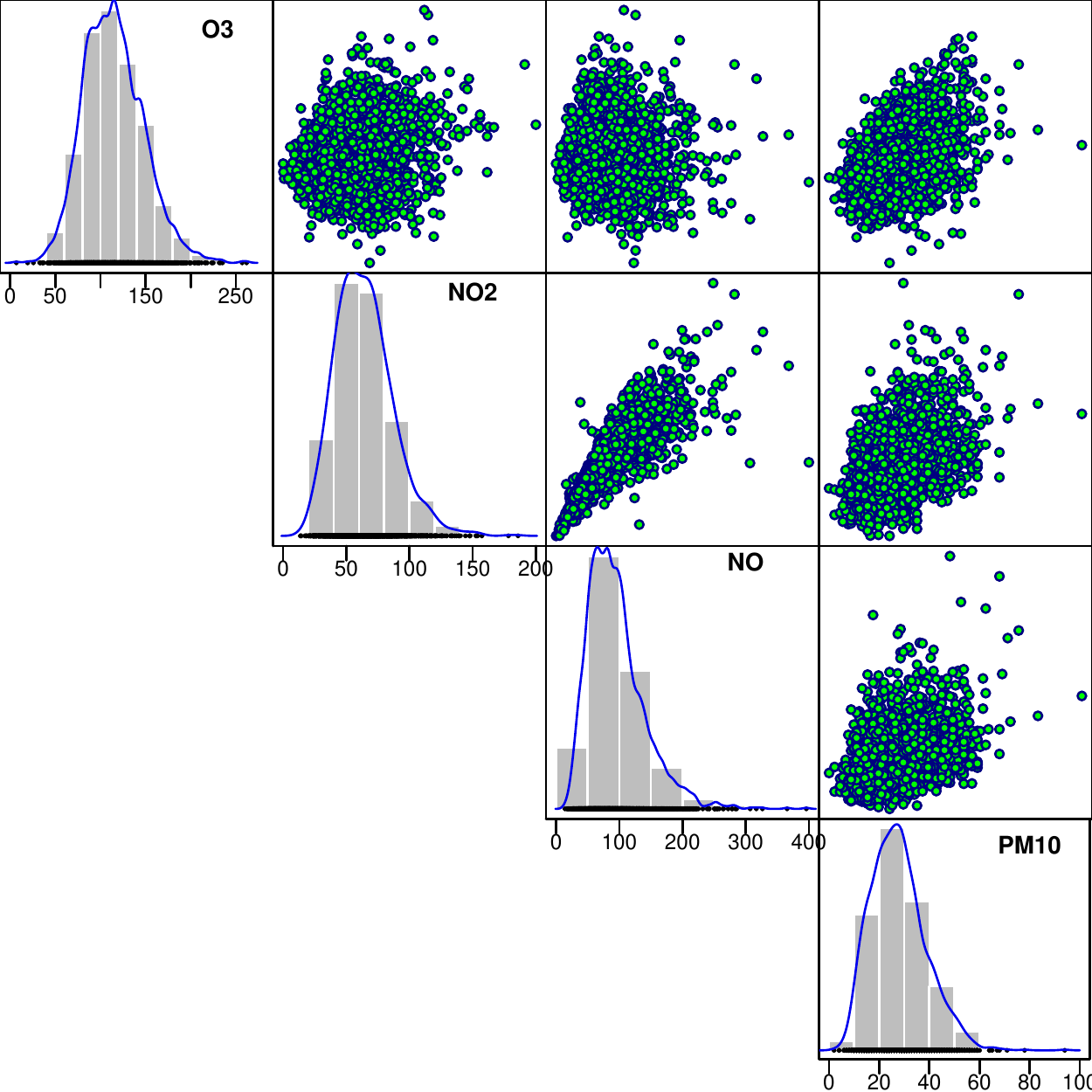}
\includegraphics[width=0.6\textwidth, page=2]{images/scatter_plots.pdf}
\vspace{7pt}
\caption{Histograms and pairwise plots of pollutants levels in $\mu$g/m$^3$. Upper and lower panels concern the summer and winter
data, respectively.}
\label{fig:scatter_plots}
\end{figure}
Figure \ref{fig:scatter_plots} displays in the top and bottom parts the pairwise scatter plot for the summer and winter datasets, respectively, together with histograms of the individual pollutants levels. SO$_2$ has been removed from the summer dataset and O$_3$ winter dataset, because they seem independent from the other pollutants. In each dataset
the pollutants seem to be highly dependent and this is especially true for NO$_2$, NO and PM$_{10}$.
In summer, O$_ 3$ is moderately dependent to NO$_2$ and PM$_{10}$.
Finally, we see that in the winter season NO$_2$, NO and PM$_{10}$ reach much higher pollution concentrations than in summer.
\begin{table}[t!]
\begin{center}
 \begin{tabular}{clclccccc}
\toprule
Scenario & Season & $n$ &   & \multicolumn{4}{c}{Pollutant} & JEEP\\
\cmidrule(lr){5-8}
&Summer & 1655 && O$_3$ & NO$_2$ & NO & PM$_{10}$ & \\
\mycmidrule(lr){1-9}
1 && & Threshold  & 120 & 200 & 200 & 50 & 0\\
\cdashlinelr{4-8}
& && MEEP & 40.181 & 0 & 2.961 & 3.444 & \\
\cmidrule(lr){1-9}
2 && & Threshold & 180 & 200 & 200 & 50 & 0\\
\cdashlinelr{4-8}
& && MEEP & 3.263 & 0 & 2.961 & 3.444 & \\
\cmidrule(lr){1-9}
3 && & Threshold & 240 & 400 & 400 & 150 & 0\\
\cdashlinelr{4-8}
& && MEEP & 0 & 0 & 0 & 0 & \\
\mycmidrule(lr){1-9}
& Season & $n$ & & \multicolumn{4}{c}{Pollutant} & \\
\cmidrule(lr){5-8}
& Winter & 1713 & & SO$_2$ & NO$_2$ & NO & PM$_{10}$ & \\
\mycmidrule(lr){1-9}
4 &&& Threshold & 125 & 200 & 200 & 50 & 0.0584\\
\cdashlinelr{4-8}
& &&MEEP  & 0.350 & 1.459 & 78.167 & 58.785 & \\
\cmidrule(lr){1-9}
5 && &Threshold & 125 & 200 & 400 & 150 & 0.0584\\
\cdashlinelr{4-8}
& &&MEEP & 0.350 & 1.459 & 32.399 & 1.926 & \\
\cmidrule(lr){1-9}
6 & &&Threshold & 125 & 200 & 800 & 150 & 0.0584\\
\cdashlinelr{4-8}
& &&MEEP & 0.350 & 1.459 & 3.853 & 1.926 & \\
\cmidrule(lr){1-9}
7 & && Threshold & 125 & 400 & 800 & 150 & 0\\
\cdashlinelr{4-8}
& && MEEP & 0.350 & 0 & 3.853 & 1.926 & \\
\bottomrule
\end{tabular}
\caption{Marginal and joint empirical probability of threshold exceedances for different combinations of thresholds.}\label{tab:scenarios}
\end{center}
\end{table}

Table \ref{tab:scenarios} reports 7 possible combinations of the thresholds listed in Table \ref{tab:thresholds}. The first 3 scenarios concern the summer season
(with $n=1655$ observations) and the last four (with $n=1713$ observations) the winter season. For each scenario the Joint Empirical
Exceedance Probability (JEEP) and for each individual pollutant of certain scenario the Marginal Empirical
Exceedance Probability (MEEP) are reported (in percentage format).
In summer, with O$_3$ approximately 40\% of observations  exceed the Limit threshold ($120$ $\mu$g/m$^3$).
Similarly, in winter, with NO approximately 78\% and 32\% of the observations exceed the Limit and Alert thresholds (200 and 400 $\mu$g/m$^3$).
Also, with PM approximately 59\% of the observations  exceed the Limit threshold (50 $\mu$g/m$^3$). Therefore, these thresholds can not be considered extreme values.
On the contrary, all the other thresholds can be considered extreme values, since that only a few observations exceed such pollutants concentrations.
In particular, with NO we found that an extreme concentration is 800 $\mu$g/m$^3$, i.e. 2 times the Alert threshold.  We estimate the probability of joint
exceedances.

We use our approach to estimate the probabilities of joint exceedances that are concerning extreme thresholds.
For this purpose, in the first place we estimate for each pollutant the probability, say $p_0$, of being below an extreme threshold, say $y$. 
We do this using the piecing together approach \citep[Chapter 2.7]{fahure10}. In short, we find a high-threshold, say $s$, with which we can use the survival function of the univariate GPD to approximate the exceeding probability of $y$, given that the latter is greater than $s$. We multiply an estimate of such a probability for the probability of exceeding $s$ (which we estimate by the empirical survival function) obtaining an estimate for the unconditional probability of exceeding $y$ (which allows to an estimate of the unconditional probability of being below than $y$).
\begin{table}[t!]
\begin{center}
 \begin{tabular}{ccccccccc}
\toprule
Season & Pollutant & Threshold & \multicolumn{5}{c}{GPD Estimates}  & \\
\cmidrule(lr){4-9}
& & & $s$ & NE & EEP  & $\hat{\sigma}$ & $\hat{\xi}$ & $\hat{p}_0$\\
\midrule
Summer & O$_3$ & 180 & 150 & 226 & 13.656 & 21.860 & -0.114  & 96.930\\
&&&& & (0.844) & (1.936) & (0.058) &  (0.804)\\
\cdashlinelr{3-9}
&& 240&& & &  &   & 99.947\\
&&&& &  &  & & (0.091)\\
\cmidrule(lr){2-9}
 & NO$_2$ & 200 & 96 & 136 &  8.218 & 14.870 & 0.067  & 99.973\\
&&&& & (0.675) & (1.862) & (0.091) &  (0.048)\\
\cdashlinelr{3-9}
&& 400&& & &  &   & 99.999\\
&&&& &  &  & & (0.0001)\\
\cmidrule(lr){2-9}
& NO & 200 & 150 & 176 &  10.634 & 36.401 & 0.047  & 97.191\\
&&&& & (0.758) & (3.970) & (0.079) &  (0.822)\\
\cdashlinelr{3-9}
&& 400&& & &  &   & 99.972\\
&&&& &  &  & & (0.048)\\
\cmidrule(lr){2-9}
& PM$_{10}$ & 50 & 47 & 89 &  5.378 & 6.387 & 0.041  & 96.623\\
&&&& & (0.554) & (0.977) & (0.110) &  (2.391)\\
\cdashlinelr{3-9}
&& 150&& & &  &   & 99.999\\
&&&& &  &  & & (0.0002)\\
\midrule
Winter & SO$_2$ & 125 & 40 & 233 & 13.602 & 24.131 & -0.026  & 99.662\\
&&&& & (0.843) & (2.206) & (0.064) &  (0.258)\\
\cmidrule(lr){2-9}
 & NO$_2$ & 200 & 130 & 240 &  14.011 & 24.376 & 0.192  & 98.576\\
&&&& & & (0.853) & (0.077) &  (0.517)\\
\cdashlinelr{3-9}
&& 400&& & &  &   &99.963\\
&&&& &  &  & & (0.046)\\
\cmidrule(lr){2-9}
 & NO & 800 & 600 & 206 &  12.026 & 195.61 & -0.029  & 95.741\\
&&&& & (0.800) & (18.991) & (0.068) &  (0.978)\\
\cmidrule(lr){2-9}
& PM$_{10}$ & 150 & 100 & 238 &  13.894 & 28.222 & -0.023  & 97.722\\
&&&& & (0.850) & (2.558) & (0.063) &  (0.684)\\
\bottomrule
\end{tabular}
\caption{Estimate of the GPD parameters and the unconditional probability to be below the individual extreme thresholds.}\label{tab:marginal_estimation}
\end{center}
\end{table}
We select the threshold $s$ through the commonly used exploratory graphical methods that are described in \citet[Chapters 4.3.1, 4.3.4]{coles2001}.
The GPD parameters are estimated using the maximum likelihood method (\citet[Chapter 4.3.2]{coles2001}).
Estimates of the variances for the GPD parameters estimates are are obtained using the asymptotic variance, see \citet{smith1984}.
An estimate of the variance for the estimate of the probability $p_0$ is obtained using the delta method (\citet[Chapter 3]{van2000}).
\begin{table}[t!]
\begin{center}
 \begin{tabular}{ccccccc}
\toprule
Scenario & (O$_3$, NO$_2$, NO, PM$_{10}$) & $t_0$ & $\hat{p}_n$ & $\hat{q}_n$ & LB-CI & UB-CI\\
\midrule
2 & (180, 200, 200, 50) & 4.6606 & 0.6042 & 0.0282 & 0.0135  & 0.0517\\
 & ( $\quad$, 200, 200, 50) & 3.3894& 0.8459 & 0.0287 & 0.0157 &  0.0480 \\
\rowcolor{lightgray} & (180, $\quad$ , 200, 50) & 44.561 & 0.4834 & 0.2154 & 0.0931  & 0.4234\\
 & (180, 200, $\quad$ , 50) & 4.7789 & 0.5438 & 0.0260 & 0.0119  & 0.0492\\
 & (180, 200, 200, $\quad$) & 3.3901 & 0.7855 & 0.0266 & 0.0142  & 0.0454\\
 & (180, 200, $\quad$ , $\quad$) & 4.0065 & 0.6647 & 0.0266 & 0.0133  & 0.0475\\
\rowcolor{lightgray} & (180, $\quad$ , 200, $\quad$) & 39.4044 & 0.7251 & 0.2857 & 0.1478  & 0.4977\\
\rowcolor{lightgray} & (180, $\quad$ , $\quad$ , 50) & 26.1881  & 3.6858 &  0.9652 &  0.7413 & 1.2334\\
 & ($\quad$, 200, 200, $\quad$ ) & 2.9383 & 0.9063 & 0.0266  & 0.0149  & 0.0438\\
 & ( $\quad$ , 200, $\quad$ , 50) & 3.3901 & 0.7855 & 0.0266 & 0.0142  & 0.0454\\
\rowcolor{lightgray} & ( $\quad$ , $\quad$ , 200, 50) & 37.2618 & 1.5710 & 0.5854 & 0.3833  & 0.8546\\
\cdashlinelr{1-7}
3 & (240, 400, 400, 150) & 0.3435 &0 & 0 & 0 & 0.0008\\
\midrule
Scenario & (SO$_2$, NO$_2$, NO, PM$_{10}$) & $t_0$ & $\hat{p}_n$ & $\hat{q}_n$ & LB-CI & UB-CI\\
\midrule
5 & (125, 200, 400, 150) & 76.2093 & 0.1751 & 0.1335 & 0.0275 & 0.3894\\
\cdashlinelr{1-7}
6 & (125, 200, 800, 150) & 7.6186 & 1.5178 & 0.1156 & 0.0757 & 0.1688 \\
\rowcolor{lightgray} & ( $\quad$ , 200, 800, 150) & 51.3463 & 1.3427 & 0.6894 & 0.4380 & 1.0310 \\
 & (125, $\quad$ , 800, 150) & 24.0850 & 0.7589 & 0.1828& 0.0975 & 0.3117 \\
 & (125, 200, $\quad$ , 150) & 22.0570 & 0.5838 & 0.1288& 0.0618 & 0.2362 \\
 & (125, 200, 800, $\quad$ ) & 38.6023 & 0.3503 & 0.1352& 0.0497 & 0.2937 \\
 & (125, 200, $\quad$ , $\quad$ ) & 7.3745 & 1.9848 & 0.1464 & 0.1016 & 0.2037 \\
\rowcolor{lightgray} & (125, $\quad$ , 800, $\quad$ ) & 4.2888 & 7.8809 & 0.3380 & 0.2852 & 0.3971 \\
 & (125, $\quad$ , $\quad$ , 150) & 7.0823 & 2.7437 & 0.1943& 0.1433 & 0.2572 \\
\rowcolor{lightgray}& ( $\quad$ , 200, 800, $\quad$ ) & 4.2635 & 33.3917 & 1.4236 & 1.3285 & 1.5213 \\
\rowcolor{lightgray}& ( $\quad$ , 200, $\quad$ , 150) & 14.5064 & 5.5458 & 0.8045& 0.6542 & 0.9773 \\
\rowcolor{lightgray} & ( $\quad$ , $\quad$ , 800, 150) & 22.4274 & 5.6042 & 1.2569 & 1.0233 & 1.5253 \\
\cdashlinelr{1-7}
7 & (125, 400, 800, 150) & 4.8869 & 0.1751 & 0.0086 & 0.0018 & 0.0250\\
\bottomrule
\end{tabular}
\caption{Probability estimates of joint exceedances of  extreme thresholds. Results are given in percentage format.}\label{tab:est_joint_exc_prob}
\end{center}
\end{table}
Note that each $p_0$ acts as a component of $\bfx_0$ in \eqref{eq:exc_prob_unif}.
Table \ref{tab:marginal_estimation} shows the estimation results. Specifically, the column named Threshold reports the extreme thresholds of the scenarios in Table \ref{tab:scenarios} with small percentages of exceedances. $s$ indicates the threshold used for estimating the univariate GPD parameters. NE is the number of exccedances of $s$ and EEP is the relative empirical exceedance probability (in percentage format). The values $\hat{\sigma}$ and $\hat{\xi}$ are the estimates of the scale and shape parameters of the univariate GPD, see equation \eqref{eq:uni_GPD}. The value $\hat{p}_0$ is an estimate of the unconditional probability (in percentage) to be below the extreme threshold reported in the third column (from the left). The standard errors are reported in parentheses. The variance of EEP is obtained using the fact that NE follows a Binomial distribution with unknown exceedance probability (estimated by EEP) and sample size $n$ (see Table \ref{tab:scenarios}).

Once  the extreme thresholds were transformed to values in (0,1), we apply the estimation method  introduced in Section  \ref{sec:estimaton of exceedance probability} for estimating the probabilities of their exceedances on the copula level, using the empirical copula of the original data. Estimation of joint exceedance probabilities on the copula level can be based on the transformation of the margins if their df are known. It was, however, shown in \citet{buecher2012} that it is more efficient if the additional knowledge of the margins is ignored and estimators are based on ranks, i.e., if the empirical copula of the initial data is used. 

Table \ref{tab:est_joint_exc_prob} reports, in the column labeled by $\hat{q}_n$, the estimates of exceedances probabilities (in percentage) for the scenarios listed in Table \ref{tab:scenarios}. The lower and upper bounds of their 95\% confidence interval are reported in the columns LB-CI and UB-CI, respectively. Due to the reasonably large underlying sample size, Clopper-Pearson is used for the confidence bounds. The factors $t_0$ are given in percentage format as well.
Furthermore, estimates for some combinations of three and two extreme thresholds are also reported. The lines highlighted in grey concern the higher estimated probabilities.
Scenarios 1 and 4 are not considered because the thresholds for O$_3$ (in summer) and NO$_2$ and NO (in winter) are not extreme. However, upper bounds for those probabilities are given by the results listed in the second and twentieth line. 
Note that in scenario 3 we found a critical value $\bfu_0$ in \eqref{eq:x0_u0_relation} such that $\hat{p}_n=0$. 
By defining the new critical value $\tilde{\bfx}_0\cdot 100 \%=(99.947,99.947,99.947,99.947)$, which uses the only the smallest $\hat{p}_0$, we were able to check condition \eqref{eqn:characteristic property of GPC}. Indeed, we found that it holds. Thus, although the exceedance probability estimate is $\hat{q}_n=0$ we computed the upper bound of its 95\% confidence interval.
Some interpretations are as follows.
In summer, we expect  that the Information and Limit thresholds for O$_3$ and PM$_{10}$, respectively, are simultaneously exceeded on average approximately
between two and four times every three years (with the latter that also means once per year).
In winter, we expect  that the Limit, double the Alert and the Alert thresholds for NO$_2$, NO and PM$_{10}$, respectively, are simultaneously exceeded on average approximately between once every two years and once per year. Finally, we expect that double the Alert and the Alert thresholds for NO$_2$ and NO, respectively, are simultaneously exceeded on average approximately between once and twice per year.
Although joint thresholds exceedances do not happen often, they should not happen at all since the involved thresholds mean indeed very extreme pollution concentrations.

\section*{Acknowledgements}

The authors thank the Italian government agency ARPA for having made available the air pollution data. The work was completed while the second author was visiting Wuerzburg University. He found a nice and stimulating environment. He thanks colleagues and staff for their hospitality which has been greatly appreciated.

\bibliographystyle{chicago}
\bibliography{evt}

\end{document}